\documentclass[12pt]{amsart}

\usepackage{graphicx,amssymb,latexsym,amsfonts,txfonts,amsmath,amsthm}
\usepackage{pdfsync,color,tabularx,rotating}
\usepackage[all,cmtip]{xy}

\usepackage{tikz}

\advance\hoffset by -0.9 truecm

\newtheorem{theo}{Theorem}
\newtheorem{prop}[theo]{Proposition}
\newtheorem{cor}[theo]{Corollary}
\newtheorem{lemma}[theo]{Lemma}
\newtheorem{conj}[theo]{Conjecture}

\theoremstyle{remark}

\theoremstyle{remark}

\theoremstyle{remark}

\begin{document}

\title{Maximal subgroups of the modular and other groups}

\author{Gareth A. Jones}

\address{School of Mathematical Sciences, University of Southampton, Southampton SO17 1BJ, UK}
\email{G.A.Jones@maths.soton.ac.uk}

\maketitle


\begin{abstract}
In 1933 B.~H.~Neumann constructed uncountably many subgroups of ${\rm SL}_2(\mathbb Z)$ which act regularly on the primitive elements of $\mathbb Z^2$. As pointed out by Magnus, their images in the modular group ${\rm PSL}_2(\mathbb Z)\cong C_3*C_2$ are maximal nonparabolic subgroups, that is, maximal with respect to containing no parabolic elements.
We strengthen and extend this result by giving a simple construction using planar maps to show that for all integers $p\ge 3$, $q\ge 2$ the triangle group $\Gamma=\Delta(p,q,\infty)\cong C_p*C_q$ has uncountably many conjugacy classes of nonparabolic maximal subgroups. We also extend results of Tretkoff and of Brenner and Lyndon for the modular group by constructing uncountably many conjugacy classes of such subgroups of $\Gamma$ which do not arise from Neumann's original method. These maximal subgroups are all generated by elliptic elements, of finite order, but a similar construction yields uncountably many conjugacy classes of torsion-free maximal subgroups of the Hecke groups $C_p*C_2$ for odd $p\ge 3$. Finally, an adaptation of work of Conder yields  uncountably many conjugacy classes of maximal subgroups of $\Delta(2,3,r)$ for all $r\ge 7$.
\end{abstract}


\section{Introduction}\label{intro}

In response to a question of Schmidt concerning the foundations of geometry, B.~H.~Neumann~\cite{Neu33} constructed uncountably many subgroups of ${\rm SL}_2({\mathbb Z})$ acting regularly on the primitive elements of ${\mathbb Z}^2$ (those with coprime coordinates, or equivalently the members of bases for ${\mathbb Z}^2$). Magnus~\cite{Mag73} (see also~\cite[\S III.2]{Mag74}) showed that their images in the modular group $\Gamma={\rm PSL}_2({\mathbb Z})$ are what he called {\sl Neumann subgroups}, those complemented by the maximal parabolic subgroup $P$ generated by the M\"obius transformation $t\mapsto t+1$, which implies that they are maximal nonparabolic subgroups, that is, maximal with respect to containing no parabolic elements of $\Gamma$. Magnus conjectured in~\cite{Mag73} that Neumann had constructed all the maximal nonparabolic subgroups of $\Gamma$, but subsequently C.~Tretkoff~\cite{Tre} produced further examples of Neumann subgroups not arising from Neumann's construction, while Brenner and Lyndon~\cite{BL83MA} (see also~\cite{Lyn}) found further examples of maximal nonparabolic subgroups of $\Gamma$ which are not Neumann subgroups.

In algebraic map theory~\cite{Jon79, JS} subgroups of $\Gamma$ correspond to trivalent (or triangular) oriented maps; nonparabolic subgroups correspond to trivalent maps with no finite faces, and among these, Neumann subgroups correspond to those with a single (infinite) face. In this paper simple constructions of maps will be used to extend the above results by producing uncountably many conjugacy classes of nonparabolic maximal subgroups, that is, subgroups which are both nonparabolic and maximal (as in fact most of Neumann's are), in a much wider class of hyperbolic triangle groups:

\begin{theo}\label{CpCqnonpara}
For each pair of integers $p\ge 3$ and $q\ge 2$ the triangle group $\Gamma=\Delta(p,q,\infty)\cong C_p*C_q$ has uncountably many conjugacy classes of nonparabolic maximal subgroups.
\end{theo}

Of course, if $p, q$ and $r$ are all integers then every subgroup of $\Delta(p,q,r)$ is nonparabolic, since this group, being cocompact, has no parabolic elements. 

Neumann subgroups and their maximality properties have also been studied geometrically by Kulkarni in~\cite{Kul}. In one sense he does so in a wider context, since he takes $\Gamma$ to be a free product of an arbitrary finite number of finite cyclic groups. On the other hand, his main theorem on maximality requires all these cyclic groups to have prime order, which is not a requirement here.

After some preliminary results in \S\ref{permutations}, we will prove Theorem~\ref{CpCqnonpara}, dealing first with the Hecke groups $\Delta(p,2,\infty)\cong C_p*C_2$ in \S\ref{proofq=2}, and then with the groups $\Delta(p,q,\infty)\cong C_p*C_q$ for $p, q\ge 3$ in \S\ref{proofq>2}. The nonparabolic maximal subgroups constructed there are all free products of cyclic groups of order $p$ or $q$, but in \S\ref{structure} we will use a similar method to construct uncountably many conjugacy classes of torsion-free maximal subgroups, each freely generated by an infinite set of parabolic elements, in the Hecke groups $\Delta(p,2,\infty)$ for all odd $p\ge 3$. In \S\ref{Neumannrev}, \S\ref{Tretkoffrev} and \S\ref{B&Lrev} we will revisit the constructions by Neumann, by Tretkoff and by Brenner and Lyndon, reinterpreting them in terms of maps, and showing how their results can be extended to other hyperbolic triangle groups $\Delta(p,q,\infty)$. In \S\ref{conseq} we will briefly consider some consequences and generalisations to other groups, such as cocompact hyperbolic groups $\Delta(p,q,r)$.


\section{Neumann permutations, subgroups and maps}\label{permutations}

Let $p$ and $q$ be integers such that $p\ge 3$ and $q\ge 2$. We define a {\sl Neumann permutation\/} of type $(p,q)$ to be a permutation $y$ of $\mathbb Z$ such that
\[(yz)^p=y^q=1,\]
where $z$ is the translation $i\mapsto i+1$, and $1$ denotes the identity permutation. For the case $p=3, q=2$ see~\cite{Neu33}, and also~\cite{Mag73} and~\cite[Lemma~3.6]{Mag74} where the function $f$ plays the role of $y$.

For $p\ge 3$ and $q\ge 2$ the triangle group
\[\Gamma=\Delta(p,q,\infty)=\langle X, Y, Z\mid X^p=Y^q=XYZ=1\rangle\cong C_p*C_q\]
is a group of orientation-preserving isometries of the hyperbolic plane $\mathbb H$, where $X$ and $Y$ are elliptic elements (with a unique fixed point in $\mathbb H$) and $Z$ is parabolic (with a unique fixed point in $\partial\mathbb H$). Extending Magnus's definition in~\cite{Mag73, Mag74} for subgroups of the modular group $\Delta(3,2,\infty)$, let us define a {\sl Neumann subgroup\/} of $\Gamma$ to be a subgroup $M$ which complements the maximal parabolic subgroup $P=\langle Z\rangle$ of $\Gamma$, that is, $\Gamma=MP$ and $M\cap P=1$.

We define a {\sl Neumann map of type $(p,q)$} to be an infinite oriented bipartite map $\mathcal N$ with one face, with the vertices in the two parts coloured black and white, and every black or white vertex having valency dividing $p$ or $q$ respectively. (Combinatorialists may recognise this as the Walsh bipartite map~\cite{Wal} for the hypermap corresponding to this representation of $\Gamma$.) If $q=2$ we can simplify $\mathcal N$ to a map $\mathcal N^{\dagger}$ by omitting all the white vertices, leaving free or non-free edges where vertices of valency $1$ or $2$ are removed, so that the directed edges of $\mathcal N^{\dagger}$ correspond to the edges of $\mathcal N$; there is no loss of information in doing this, since $\mathcal N$ can be recovered by adding a white vertex to each edge of $\mathcal N^{\dagger}$, of valency $1$ or $2$ as the edge is free or not. 

For a given pair $p, q$ there are natural bijections between (isomorphism classes of) these three sets of objects. Given a Neumann permutation $y$ we define a permutation representation of $\Gamma$ on $\mathbb Z$ by sending $X, Y$ and $Z$ to the permutations $x:=(yz)^{-1}$, $y$ and $z$. This representation is transitive, and $P$ acts regularly, so it complements the subgroup $M$ of $\Gamma$ fixing a particular integer. This stabiliser $M$ is uniquely determined up to conjugacy. Conversely, given a Neumann subgroup $M$ of $\Gamma$, we can use the powers of $Z$ as coset representatives of $M$, so that $Z$, acting on the cosets of $M$, induces the translation $z:i\mapsto i+1$ on $\mathbb Z$, while $Y$ induces a permutation $y$ satisfying $(yz)^p=y^q=1$.

As in the more general algebraic theory of maps~\cite{JS}, a Neumann map $\mathcal N$ determines a triple of permutations $x, y, z$ of the set $\Omega$ of its edges, with $x$ and $y$ using the orientation of the surface to rotate edges around their incident black and white vertices, so that $x^p=y^q=1$, while $z:=(xy)^{-1}$ follows the orientation around the unique face, so that it has a single (infinite) cycle on $\Omega$; given any chosen edge $\alpha$ we can identify each edge $\beta=\alpha z^i\in\Omega$ with the integer $i$, so that $z$ acts as the translation $i\mapsto i+1$ on $\mathbb Z$, and $y$ is a Neumann permutation. Conversely, given a Neumann permutation $y$ one can reconstruct $\mathcal N$ from the permutations $x:=(yz)^{-1}$, $y$ and $z$ of $\mathbb Z$, with edges corresponding to elements of $\mathbb Z$, and black and white vertices corresponding to the cycles of $x$ and $y$, so that the cyclic order of incident edges determines the local orientation around each vertex.

The significance of Neumann permutations, subgroups and maps lies in the following simple result (see~\cite[Theorem 4]{Mag73} and~\cite[Theorem~3.4(i)]{Mag74} for the case $p=3, q=2$, and~\cite[Prop.~2.1]{Kul} for a more general result):

\begin{prop}
Each Neumann subgroup $M$ is a maximal nonparabolic subgroup of $\Gamma$.
\end{prop}

\begin{proof}
The parabolic elements of $\Gamma$ are the conjugates of the non-identity powers of $Z$. Since $z$ has no finite cycles on $\mathbb Z$ (equivalently, the map $\mathcal N$ has no finite faces), there are no such elements in the stabiliser $M$, which is therefore a nonparabolic subgroup of $\Gamma$. It is, in fact, maximal with respect to this property, for if a subgroup $M^*$ of $\Gamma$ properly contains $M$, it must contain a coset $Mg\ne M$ of $M$ in $\Gamma$, and hence contains the corresponding coset representative $Z^i\ne 1$, which is parabolic.
\end{proof}

Non-isomorphic Neumann maps give inequivalent permutation representations of $\Gamma$, and hence distinct conjugacy classes of stabilisers $M$. We will show that for each pair $p\ge 3$ and $q\ge 2$ there are $2^{\aleph_0}$ isomorphism classes of Neumann maps of type $(p,q)$, so we obtain $2^{\aleph_0}$ distinct conjugacy classes of maximal nonparabolic subgroups of $\Gamma$. In particular, by taking $p=3$ and $q=2$ we obtain uncountably many maximal nonparabolic subgroups of the modular group (for background on this group, see~\cite{Mag74} or~\cite[Ch.~6]{JS87}). In~\cite{Neu33} Neumann used a different method to construct such subgroups, involving a rather complicated construction of suitable permutations of $\mathbb Z$, though it is conceivable that his purely algebraic approach was originally based on combinatorial or topological ideas.

In fact, for any $p\ge 3$ and $q\ge 2$, many of the Neumann subgroups $M$ of $\Gamma$ are maximal as subgroups of $\Gamma$, not just as nonparabolic subgroups. This is equivalent to $\Gamma$ acting primitively on $\Omega$, that is, preserving no equivalence relations on $\Omega$ other than the identity and universal relations. Now we can identify $\Omega$ with $\mathbb Z$ as above, so that $Z$ acts on $\mathbb Z$ by $z:i\mapsto i+1$. Any nontrivial $\Gamma$-invariant equivalence relation $\sim$ on $\Omega$ must therefore induce a nontrivial translation-invariant equivalence relation on $\mathbb Z$, and the only possibility for this is congruence mod~$(n)$ for some integer $n\ge 2$. Since $\Gamma=\langle Y, Z\rangle$, the relation $\sim$ will then be $\Gamma$-invariant if and only if it is preserved by $y$, that is, $i\equiv j$ mod~$(n)$ implies $iy\equiv jy$ mod~$(n)$. In many cases, for each $n$ there will be some pair $i, j$ for which this implication fails, so that $\Gamma$ acts primitively and the stabilisers $M$ are maximal subgroups. Indeed, it is easy to construct Neumann maps $\mathcal N$ for which this happens. For example:

\begin{lemma}\label{notinv}
Suppose that $x$ fixes $i$ and $y$ fixes $j$, where $i\equiv j$ mod~$(n)$ for some integer $n\ge 2$. Then congruence mod~$(n)$ is not $\Gamma$-invariant.
\end{lemma}

\begin{proof}
Suppose that congruence mod~$(n)$ is $\Gamma$-invariant. Let $E$ denote the congruence class $[i]=[j]$ of $i$ and $j$ mod~$(n)$. Since $ix=i$ we have $Ex=E$, and similarly $jy=j$ implies that $Ey=E$. Since $\Gamma=\langle X, Y\rangle$ it follows that $E$ is invariant under $\Gamma$. But $\Gamma$ acts transitively on $\Omega$, and $E\ne\emptyset$, so $E=\Omega$. Thus $n=1$, against our hypothesis.
\end{proof}


\section{Proof of Theorem 1: the case $q=2$}\label{proofq=2}

By using the ideas in the preceding section we obtain a simple proof of a generalisation of Neumann's result:

\begin{cor}\label{p3}
For each integer $p\ge 3$ the Hecke group $\Gamma=\Delta(p,2,\infty)\cong C_p*C_2$ has uncountably many conjugacy classes of nonparabolic maximal subgroups.
\end{cor}

\begin{figure}[h!]
\begin{center}
\begin{tikzpicture}[scale=0.17, inner sep=0.8mm]

\node (0) at (0,0) [shape=circle, fill=black] {};
\node (1) at (10,0) [shape=circle, fill=black] {};
\node (2) at (20,0) [shape=circle, fill=black] {};
\node (3) at (30,0) [shape=circle, fill=black] {};
\node (4) at (40,0) [shape=circle, fill=black] {};
\node (5) at (50,0) [shape=circle, fill=black] {};
\node (6) at (65,0) [shape=circle, fill=black] {};
\node (7) at (75,0) [shape=circle, fill=black] {};
\node (8) at (85,0) [shape=circle, fill=black] {};

\draw [thick] (0,0) to (54,0);
\draw [thick,dashed] (54,0) to (61,0);
\draw [thick] (61,0) to (89,0);
\draw [thick,dashed] (89,0) to (93,0);
\draw [thick] (1) to (10,10);
\draw [thick] (2) to (20,10);
\draw [thick] (4) to (40,10);
\draw [thick] (7) to (75,10);

\node at (45.5,-1.5) {$6-4p$};
\node at (35.5,-1.5) {$4-3p$};
\node at (25.5,-1.5) {$3-2p$};
\node at (16,-5) {$-p$};
\node at (16,-1.5) {$1-p$};
\node at (6,-5) {$-2$};
\node at (6,-1.5) {$-1$};
\node at (2.5,1.5) {$0$};
\node at (8.5,5) {$1$};
\node at (12.5,1.5) {$2$};
\node at (18.5,5) {$3$};
\node at (22.5,1.5) {$4$};
\node at (33,1.5) {$5$};
\node at (38.5,5) {$6$};
\node at (43,1.5) {$7$};
\node at (69,1.5) {$3n-1$};
\node at (73,5) {$3n$};
\node at (79,1.5) {$3n+1$};

\draw [thick] (7,-10) to (1) to (13,-10);
\draw [dashed] (7.7,-10) to (12.3,-10);
\node at (10,-13) {$(p-3)$};

\draw [thick] (17,-10) to (2) to (23,-10);
\draw [dashed] (17.7,-10) to (22.3,-10);
\node at (20,-13) {$(p-3)$};

\draw [thick] (27,-10) to (3) to (33,-10);
\draw [dashed] (27.7,-10) to (32.3,-10);
\node at (30,-13) {$(p-2)$};

\draw [thick] (37,-10) to (4) to (43,-10);
\draw [dashed] (37.7,-10) to (42.7,-10);
\node at (40,-13) {$(p-3)$};

\draw [thick] (47,-10) to (5) to (53,-10);
\draw [dashed] (47.5,-10) to (52.5,-10);
\node at (50,-13) {$(p-2)$};

\draw [thick] (62,-10) to (6) to (68,-10);
\draw [dashed] (62.7,-10) to (67.5,-10);
\node at (65,-13) {$(p-2)$};

\draw [thick] (72,-10) to (7) to (78,-10);
\draw [dashed] (72.7,-10) to (77.7,-10);
\node at (75,-13) {$(p-3)$};

\draw [thick] (82,-10) to (8) to (88,-10);
\draw [dashed] (82.7,-10) to (87.7,-10);
\node at (85,-13) {$(p-2)$};

\node at (57.5,8) {$F$};

\end{tikzpicture}
\end{center}
\caption{A $p$-valent Neumann map $\mathcal N_p$ for $p\ge 3$} 
\label{mapNp}
\end{figure}
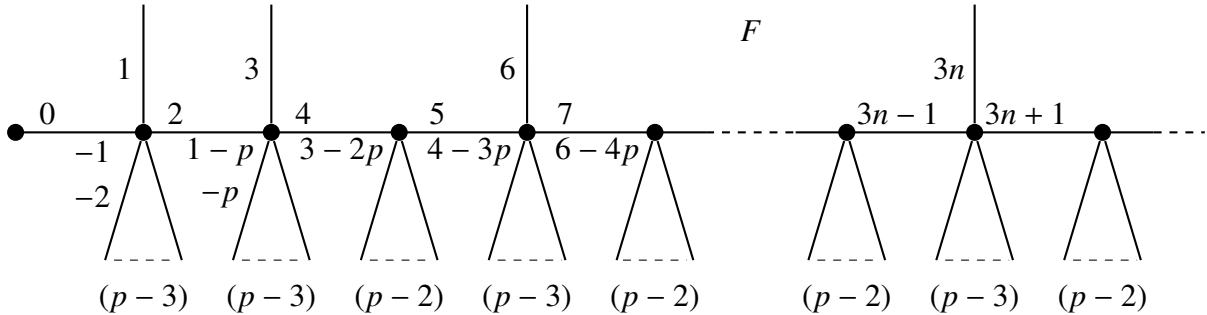

\begin{proof}
 Let $\mathcal N_p$ be the Neumann map of type $(p,2)$ shown in Figure~\ref{mapNp}, where white vertices have been omitted as explained earlier, and the numbers in parentheses indicate how many free edges there are in each `fan'. The pattern repeats periodically to the right. The leftmost directed edge, with the unique $1$-valent vertex as its target, has been chosen as $\alpha$, so that each directed edge $\beta=\alpha z^i\in\Omega$ is labelled with the integer $i$. (To save space in the diagram, only a few significant labels are shown.) For each $n\ge 2$ we can apply Lemma~\ref{notinv} to the directed edges labelled $i=0$ and $j=3n$, fixed by $x$ and $y$ respectively, to show that congruence mod~$(n)$ is not a $\Gamma$-invariant relation. It follows that the representation of $\Gamma$ is primitive, and the stabiliser $M=\Gamma_{\alpha}$, together with its conjugates $\Gamma_{\beta}$ for $\beta\in\Omega$, is a nonparabolic maximal subgroup of $\Gamma$.

In order to produce $2^{\aleph_0}$ conjugacy classes of such subgroups we can modify $\mathcal N_p$ by adding $1$-valent black vertices to an arbitrary subset of the free edges with negative labels (those below the horizontal axis): this adds extra directed edges (fixed by $x$) to $\Omega$, and changes the labelling below the axis, but it has no effect on the labelling above the axis, so the preceding proof still applies.
\end{proof}

In fact, there are many Neumann maps of type $(p,2)$ giving rise to uncountably many conjugacy classes of nonparabolic maximal subgroups of $\Gamma$, after suitable addition or deletion of $1$-valent black vertices. Define an edge of a Neumann map $\mathcal N$ of type $(p,2)$ to be {\sl terminal\/} if it is a free edge or is incident with a vertex of valency $1$. Define two such maps to {\sl have the same shape\/} if they differ only by the addition or deletion of vertices of valency $1$ on terminal edges.

\begin{prop}\label{shape}
Given any shape of Neumann maps of type $(p,2)$ with infinitely many terminal edges, there are uncountably many Neumann maps of that shape for which the corresponding nonparabolic subgroups of $\Gamma=\Delta(p,2,\infty)$ are maximal in $\Gamma$.
\end{prop}

\begin{proof}
Let $\mathcal N_0$ be a Neumann map of the given shape, with just one $1$-valent vertex $v_0$. Let $\alpha$ be the directed edge with target $v_0$, and use $\alpha$ to identify the set $\Omega$ of directed edges of $\mathcal N_0$ with $\mathbb Z$ as above. Since there are infinitely many terminal edges, $\mathcal N_0$ has infinitely many free edges,  so there must be infinitely many with label $i>0$, or infinitely many with label $i<0$, or both; replacing $\mathcal N_0$ with its mirror image if necessary (since the conclusion is invariant under reflection) we may assume the former.

We will now add $1$-valent vertices to positively labelled free edges of $\mathcal N_0$, thus changing labels, to ensure that for each $n\ge 2$ there is a free edge labelled with some multiple of $n$, so that Lemma~\ref{notinv} shows that congruence mod~$(n)$ is not $\Gamma$-invariant. Adding a $1$-valent vertex to a free edge labelled $i>0$ adds a directed edge labelled $i+1$ to $\Omega$, and increases all existing labels $j>i$ by $1$, leaving all labels $j\le i$ unchanged. By doing this to at most one free edge with label $i>0$ we can produce a free edge with label $j_2$ divisible by $2$. Then by doing this to at most two free edges with labels $i>j_2$ we can produce a free edge with label $j_3$ divisible by $3$. Continuing in this way we produce a Neumann map $\mathcal N$, of the same shape as $\mathcal N_0$, with the property that for every $n\ge 2$ there is a free edge labelled with some multiple $j_n$ of $n$. Then Lemma~\ref{notinv}, with $i=0$ fixed by $x$ and $j=j_n$ fixed by $y$, shows that congruence mod~$(n)$ is not $\Gamma$-invariant, so that $\Gamma$ acts primitively on $\Omega$ and the subgroup $M=\Gamma_{\alpha}$ is maximal. Moreover, at each stage of this process, in producing $j_n$, we have a choice of which free edges should receive $1$-valent vertices, so there are $2^{\aleph_0}$ non-isomorphic possibilities for $\mathcal N$, giving $2^{\aleph_0}$ conjugacy classes of maximal subgroups $M$.
\end{proof}

Here it is necessary to include the hypothesis that the shape has infinitely many terminal edges, since there are Neumann maps with only finitely many, or even none, as we shall see in \S\ref{Tretkoffrev}. However, it is easy to see (for instance, by considering finite subgraphs) that every planar Neumann map has infinitely many.


\section{Proof of Theorem~\ref{CpCqnonpara}: the case $q\ge 3$}\label{proofq>2}

Having dealt with the case $q=2$ of Theorem~\ref{CpCqnonpara} in the preceding section, we now deal with the case $q\ge 3$. Since $q\ne 2$ we will use bipartite Neumann maps, with vertices coloured black and white, and take $\Omega$ to be the set of edges, rather than directed edges.

\begin{theo}\label{q3}
For each pair of integers $p\ge 3$ and $q\ge 3$ the triangle group $\Gamma=\Delta(p,q,\infty)$ has uncountably many conjugacy classes of nonparabolic maximal subgroups.
\end{theo}

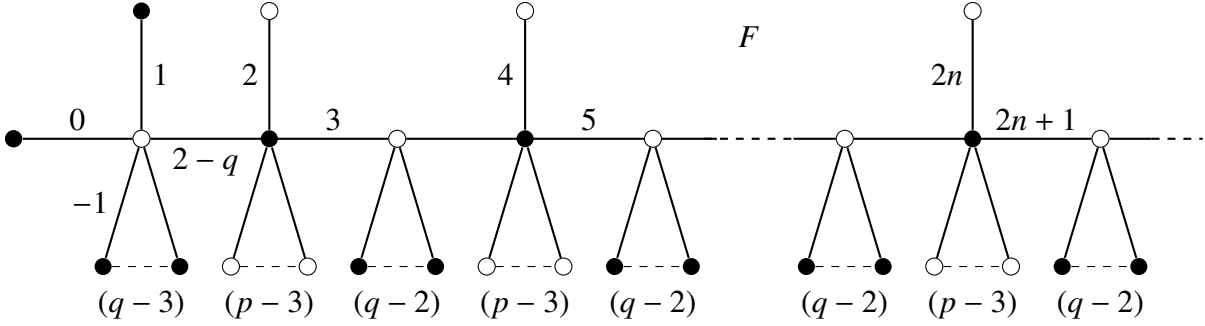
\begin{figure}[h!]
\begin{center}
\begin{tikzpicture}[scale=0.17, inner sep=0.8mm]

\node (0) at (0,0) [shape=circle, fill=black] {};
\node (1) at (10,0) [shape=circle, draw] {};
\node (2) at (20,0) [shape=circle, fill=black] {};
\node (3) at (30,0) [shape=circle, draw] {};
\node (4) at (40,0) [shape=circle, fill=black] {};
\node (5) at (50,0) [shape=circle, draw] {};
\node (6) at (65,0) [shape=circle, draw] {};
\node (7) at (75,0) [shape=circle, fill=black] {};
\node (8) at (85,0) [shape=circle, draw] {};

\node (0) at (0,0) [shape=circle, fill=black] {};
\node (1+) at (10,10) [shape=circle, fill=black] {};
\node (2+) at (20,10) [shape=circle, draw] {};
\node (4+) at (40,10) [shape=circle, draw] {};
\node (7+) at (75,10) [shape=circle, draw] {};

\node (1l) at (7,-10) [shape=circle, fill=black] {};
\node (1r) at (13,-10) [shape=circle, fill=black] {};
\node (2l) at (17,-10) [shape=circle, draw] {};
\node (2r) at (23,-10) [shape=circle, draw] {};
\node (3l) at (27,-10) [shape=circle, fill=black] {};
\node (3r) at (33,-10) [shape=circle, fill=black] {};
\node (4l) at (37,-10) [shape=circle, draw] {};
\node (4r) at (43,-10) [shape=circle, draw] {};
\node (5l) at (47,-10) [shape=circle, fill=black] {};
\node (5r) at (53,-10) [shape=circle, fill=black] {};
\node (6l) at (62,-10) [shape=circle, fill=black] {};
\node (6r) at (68,-10) [shape=circle, fill=black] {};
\node (7l) at (72,-10) [shape=circle, draw] {};
\node (7r) at (78,-10) [shape=circle, draw] {};
\node (8l) at (82,-10) [shape=circle, fill=black] {};
\node (8r) at (88,-10) [shape=circle, fill=black] {};

\draw [thick] (0) to (1) to (3) to (5) to (55,0);
\draw [thick,dashed] (54,0) to (61,0);
\draw [thick] (61,0) to (6) to (8) to (89,0);
\draw [thick,dashed] (89,0) to (93,0);
\draw [thick] (1) to (1+);
\draw [thick] (2) to (2+);
\draw [thick] (4) to (4+);
\draw [thick] (7) to (7+);

\node at (15,-1.8) {$2-q$};
\node at (6,-5) {$-1$};
\node at (5,1.5) {$0$};
\node at (11.5,5) {$1$};
\node at (18.5,5) {$2$};
\node at (25,1.5) {$3$};
\node at (38.5,5) {$4$};
\node at (45,1.5) {$5$};
\node at (73,5) {$2n$};
\node at (80,1.5) {$2n+1$};

\draw [thick] (1l) to (1) to (1r);
\draw [dashed] (1l) to (1r);
\node at (10,-13) {$(q-3)$};

\draw [thick] (2l) to (2) to (2r);
\draw [dashed] (2l) to (2r);
\node at (20,-13) {$(p-3)$};

\draw [thick] (3l) to (3) to (3r);
\draw [dashed] (3l) to (3r);
\node at (30,-13) {$(q-2)$};

\draw [thick] (4l) to (4) to (4r);
\draw [dashed] (4l) to (4r);
\node at (40,-13) {$(p-3)$};

\draw [thick] (5l) to (5) to (5r);
\draw [dashed] (5l) to (5r);
\node at (50,-13) {$(q-2)$};

\draw [thick] (6l) to (6) to (6r);
\draw [dashed] (6l) to (6r);
\node at (65,-13) {$(q-2)$};

\draw [thick] (7l) to (7) to (7r);
\draw [dashed] (7l) to (7r);
\node at (75,-13) {$(p-3)$};

\draw [thick] (8l) to (8) to (8r);
\draw [dashed] (8l) to (8r);
\node at (85,-13) {$(q-2)$};

\node at (57.5,8) {$F$};

\end{tikzpicture}
\end{center}
\caption{The bipartite map $\mathcal N_{p,q}$ for $p, q\ge 3$} 
\label{mapNpq}
\end{figure}

\begin{proof}
Let $\mathcal N_{p,q}$ be the Neumann map of type $(p,q)$ shown in Figure~\ref{mapNpq}, where (as before) the integers in parentheses give the number of edges and $1$-valent vertices in each `fan'. Thus the black vertices all have valency $p$ or $1$, and the white vertices all have valency $q$ or $1$.

This map corresponds to a transitive permutation representation of $\Gamma$ on the set $\Omega$ of its edges. As in the case $q=2$, the single face $F$ corresponds to the single cycle $C$ of $z=(xy)^{-1}$ on $\Omega$; defining $\alpha$ to be the leftmost edge in Figure~\ref{mapNpq}, we can label each edge $\alpha z^i$ with the integer $i$. 

The proof that $\Gamma$ acts primitively on $\Omega$ is identical to that for $q=2$, except that we now permute edges rather than directed edges, and it is the edge labelled $2n$ and fixed by $y$ which shows that $Ey=E$. It follows that the stabiliser $M=\Gamma_{\alpha}$ of $\alpha$ in $\Gamma$ is a maximal subgroup. As before, it is nonparabolic because $z$ has no finite cycles. In order to produce $2^{\aleph_0}$ such subgroups we can modify the map $\mathcal N_{p,q}$ by adding fans of $p-1$ or $q-1$ edges and $1$-valent white or black vertices to an arbitrary set of the $1$-valent black and white vertices below the horizontal axis: the resulting expansion of $\Omega$ and redefining of negative labels have no effect on the preceding proof.
\end{proof}


\section{Structure of maximal subgroups}\label{structure}

The maximal subgroups $M$ constructed so far in this paper are all free products of cyclic groups of order $p\;(=3$ for the modular group) and $q\;(=2$ for the Hecke groups), in bijective correspondence with the fixed points of $x$ and $y$ on $\Omega$. (More generally, by the Kurosh Subgroup Theorem (see~\cite[Ch.~IV, Theorem~1.10]{LS}), {\sl any\/} subgroup of $\Gamma=C_p*C_q$ is a free product of subgroups $C_r$ for $r$ dividing $p$ or $q$ or $r=\infty$. However, by the construction of the various maps we have used, proper divisors $d$ of $p$ and $q$ do not arise as black or white vertex-valencies and hence as cycle-lengths for $x$ or $y$, so proper divisors $r=p/d$ or $q/d$ do not arise.) 

Topologically, the corresponding generators for $M$ can be seen as monodromy generators for the covering of the corresponding map by the universal map of that type on $\mathbb H$, induced by the inclusion of the identity subgroup in $M$.
Algebraically, one can see this free product structure by applying the Reidemeister--Schreier algorithm~\cite[\S II.4]{LS} to obtain a presentation for $M$.
If $M$ is any Neumann subgroup of type $(p,q)$, the elements $Z^i$ of the maximal parabolic subgroup $P$ form a Schreier transversal for $M$ in $\Gamma$. Applying the algorithm to this transversal, and eliminating redundancies, we find that each black vertex of valency $d<p$ yields a generator $Z^iX^dZ^{-i}$ for $M$ (where an incident edge is labelled $i\in{\mathbb Z}$), together with a relation that its $p/d$-th power is the identity; a similar remark applies to the white vertices, giving conjugates of powers of $Y$, and there are no further generators or relations, so we obtain the claimed free product decomposition for $M$.

A more general planar bipartite map of type $(p,q)$, with any number of faces, can be transformed into a coset diagram for $M$ in $\Gamma$ with respect to the generators $X$ and $Y$, with a vertex on each edge of the map, and directed edges showing the actions of $X$ and $Y$. The geodesics in a spanning tree for this graph, from a chosen vertex $\alpha$ to the other vertices, then yield words in $X$ and $Y$ representing a Schreier transversal for $M$ in $\Gamma$. The Reidemeister--Schreier algorithm then gives a free product decomposition as before, except that now any face of finite valency $d$, corresponding to a cycle of $z$ of length $d$, yields an additional infinite cyclic free factor, generated by a conjugate of $Z^d$.

With this idea in mind, one can use or adapt the preceding constructions to produce maximal subgroups with various other types of generating sets. For example, the conjugacy class of subgroups associated with the Neumann map $\mathcal N_p$ in Figure~\ref{mapNp} are each generated by one element of order $p$ and infinitely many of order $2$. Similar arguments show that a map of the same shape but with the fixed points of $x$ and $y$ transposed yields maximal subgroups generated by one element of order $2$ and infinitely many of order $p$. However, there are some obvious restrictions. For instance, one cannot construct maximal subgroups of $\Gamma=\Delta(p,q,\infty)$ generated entirely by elements of the {\sl same\/} finite order $k>1$: such elements are elliptic, and are conjugate to powers of $X':=X^{p/k}$ or $Y':=Y^{q/k}$ as $k$ divides $p$ or $q$ (possibly both), so any subgroup they generate must be contained in the normal closure of $X'$ and $Y'$, which is a proper subgroup of $\Gamma$ for $k>1$.

As an example of what can be achieved, we have the following:

\begin{theo}\label{Hecketorsionfree}
For each odd integer $p\ge 3$ the Hecke group $\Gamma=\Delta(p,2,\infty)$ has uncountably many conjugacy classes of torsion-free maximal subgroups, each freely generated by a countably infinite set of parabolic elements.
\end{theo}

\begin{proof}
Define $l:=(p-1)/2$, so that $l\ge 1$. Let $\mathcal N'_p$ be the planar map in Figure~\ref{mapN'p}, where each vertex on or off the horizontal axis is incident with $l-1$ or $l$ loops respectively, so that all vertices have valency $p$. This map therefore represents a transitive permutation representation $\Gamma\to G$ in which the point-stabilisers are freely generated by parabolic elements, corresponding to the loops in the map. In particular, these subgroups are torsion-free. Our aim is to modify this construction in order to produce uncountably many conjugacy classes of such subgroups, all maximal in $\Gamma$.

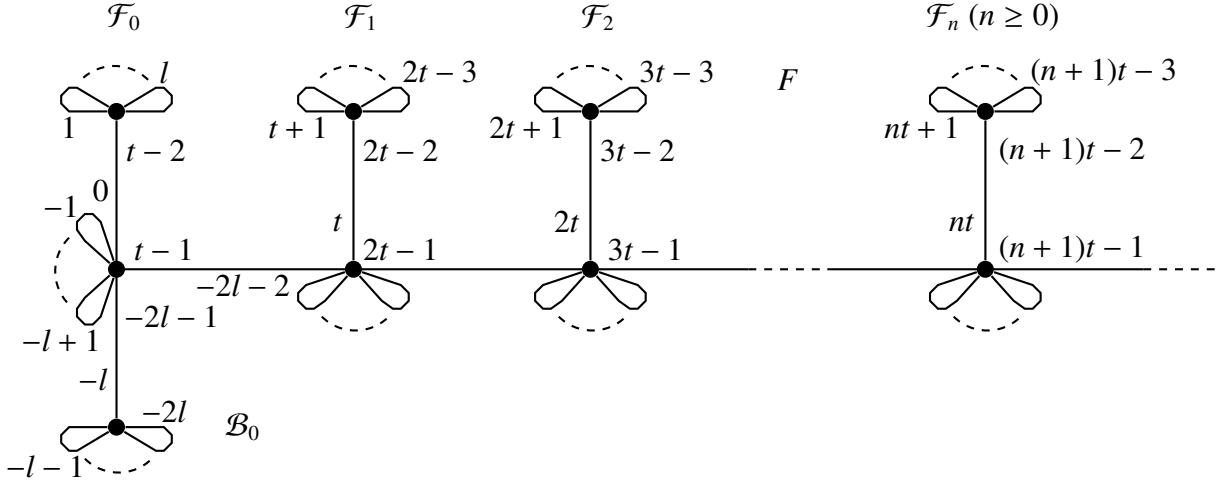
\begin{figure}[h!]
\begin{center}
\begin{tikzpicture}[scale=0.21, inner sep=0.8mm]

\node (2) at (20,0) [shape=circle, fill=black] {};
\node (3) at (35,0) [shape=circle, fill=black] {};
\node (4) at (50,0) [shape=circle, fill=black] {};
\node (7) at (75,0) [shape=circle, fill=black] {};

\node (2+) at (20,10) [shape=circle, fill=black] {};
\node (2-) at (20,-10) [shape=circle, fill=black] {};
\node (3+) at (35,10) [shape=circle, fill=black] {};
\node (4+) at (50,10) [shape=circle, fill=black] {};
\node (7+) at (75,10) [shape=circle, fill=black] {};

\draw [thick] (20,0) to (60,0);

\draw [thick,dashed] (60,0) to (65,0);
\draw [thick] (65,0) to (85,0);
\draw [thick,dashed] (85,0) to (90,0);

\draw [thick] (2-) to (2+);
\draw [thick] (3) to (3+);
\draw [thick] (4) to (4+);
\draw [thick] (7) to (7+);

\draw [thick] (2-) to (17,-10) to (16.5,-10.5) to (16.5,-11) to (17,-11.5) to (17.5,-11.5) to (2-);
\draw [thick] (2-) to (23,-10) to (23.5,-10.5) to (23.5,-11) to (23,-11.5) to (22.5,-11.5) to (2-);
\draw [thick, dashed] (18,-12) arc (225:315:3);

\draw [thick] (2+) to (17,10) to (16.5,10.5) to (16.5,11) to (17,11.5) to (17.5,11.5) to (2+);
\draw [thick] (2+) to (23,10) to (23.5,10.5) to (23.5,11) to (23,11.5) to (22.5,11.5) to (2+);
\draw [thick, dashed] (22,12) arc (45:135:3);

\draw [thick] (3+) to (32,10) to (31.5,10.5) to (31.5,11) to (32,11.5) to (32.5,11.5) to (3+);
\draw [thick] (3+) to (38,10) to (38.5,10.5) to (38.5,11) to (38,11.5) to (37.5,11.5) to (3+);
\draw [thick, dashed] (37,12) arc (45:135:3);

\draw [thick] (4+) to (47,10) to (46.5,10.5) to (46.5,11) to (47,11.5) to (47.5,11.5) to (4+);
\draw [thick] (4+) to (53,10) to (53.5,10.5) to (53.5,11) to (53,11.5) to (52.5,11.5) to (4+);
\draw [thick, dashed] (52,12) arc (45:135:3);

\draw [thick] (7+) to (72,10) to (71.5,10.5) to (71.5,11) to (72,11.5) to (72.5,11.5) to (7+);
\draw [thick] (7+) to (78,10) to (78.5,10.5) to (78.5,11) to (78,11.5) to (77.5,11.5) to (7+);
\draw [thick, dashed] (77,12) arc (45:135:3);

\draw [thick] (2) to (19,-3) to (18.5,-3.5) to (18,-3.5) to (17.5,-3) to (17.5,-2.5) to (18,-1.8) to (2);
\draw [thick] (2) to (19,3) to (18.5,3.5) to (18,3.5) to (17.5,3) to (17.5,2.5) to (18,1.8) to (2);
\draw [thick, dashed] (17,2) arc (135:225:3);

\draw [thick] (3) to (32,-1) to (31.5,-1.5) to (31.5,-2) to (32,-2.5) to (32.5,-2.5) to (33.2,-2) to (3);
\draw [thick] (3) to (38,-1) to (38.5,-1.5) to (38.5,-2) to (38,-2.5) to (37.5,-2.5) to (36.8,-2) to (3);
\draw [thick, dashed] (33,-3) arc (225:315:3);

\draw [thick] (4) to (47,-1) to (46.5,-1.5) to (46.5,-2) to (47,-2.5) to (47.5,-2.5) to (48.2,-2) to (4);
\draw [thick] (4) to (53,-1) to (53.5,-1.5) to (53.5,-2) to (53,-2.5) to (52.5,-2.5) to (51.8,-2) to (4);
\draw [thick, dashed] (48,-3) arc (225:315:3);

\draw [thick] (7) to (72,-1) to (71.5,-1.5) to (71.5,-2) to (72,-2.5) to (72.5,-2.5) to (73.2,-2) to (7);
\draw [thick] (7) to (78,-1) to (78.5,-1.5) to (78.5,-2) to (78,-2.5) to (77.5,-2.5) to (76.8,-2) to (7);
\draw [thick, dashed] (73,-3) arc (225:315:3);



\node at (28,-1) {$-2l-2$};
\node at (23.5,-3) {$-2l-1$};
\node at (23,-9) {$-2l$};
\node at (15.5,-12.5) {$-l-1$};
\node at (18.5,-7) {$-l$};
\node at (16.5,-4.5) {$-l+1$};
\node at (16.5,4) {$-1$};
\node at (19,5) {$0$};
\node at (17,9) {$1$};
\node at (23,12.5) {$l$};
\node at (22.5,7.5) {$t-2$};
\node at (23,1.25) {$t-1$};
\node at (34,3) {$t$};
\node at (31.5,9) {$t+1$};
\node at (40.5,12.5) {$2t-3$};
\node at (38,7.5) {$2t-2$};
\node at (38,1.25) {$2t-1$};
\node at (48.5,3) {$2t$};
\node at (46,9) {$2t+1$};
\node at (55.5,12.5) {$3t-3$};
\node at (53,7.5) {$3t-2$};
\node at (53.5,1.25) {$3t-1$};
\node at (73.5,3) {$nt$};
\node at (71,9) {$nt+1$};
\node at (82.5,12.5) {$(n+1)t-3$};
\node at (80.5,7.5) {$(n+1)t-2$};
\node at (80.5,1.25) {$(n+1)t-1$};

\node at (62.5,12) {$F$};
\node at (20.5,16) {${\mathcal F}_0$};
\node at (35.5,16) {${\mathcal F}_1$};
\node at (50.5,16) {${\mathcal F}_2$};
\node at (28,-10) {${\mathcal B}_0$};
\node at (75.5,16) {${\mathcal F}_n\;(n\ge 0)$};

\end{tikzpicture}
\end{center}
\caption{The $p$-valent map ${\mathcal N}'_p$, with $l=(p-1)/2$ and $t=l+3$} 
\label{mapN'p}
\end{figure}

In Figure~\ref{mapN'p} the label $0$ indicates a directed edge $\alpha$, and each other directed edge $\alpha z^i\;(i\in{\mathbb Z})$ in the cycle $C$ of $z$ containing $\alpha$ is labelled $i$. This cycle corresponds to the unbounded face $F$, while $\Omega\setminus C$ consists of fixed points of $z$, one for each loop. By analogy with gardening, we will refer to the connected subgraphs above the axis as `flowers' $\mathcal F_n$ for $n\ge 0$, and that below it as a `bulb' $\mathcal B_0$. Note that the downward directed edge in the `stem' of $\mathcal F_n$ has the label $tn$, where $t:=l+3=(p+5)/2$; we will call these the {\sl principal directed edge\/} and the {\sl principal label\/} $\lambda(n)$ of $\mathcal F_n$.

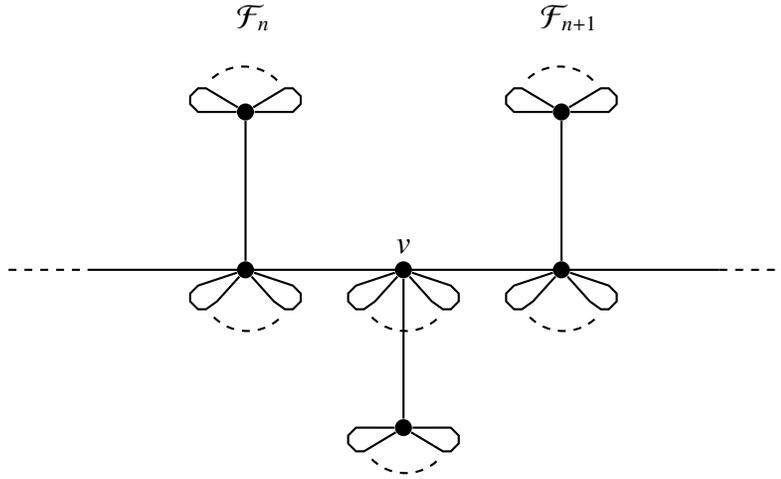
\begin{figure}[h!]
\begin{center}
\begin{tikzpicture}[scale=0.21, inner sep=0.8mm]

\node (5) at (50,0) [shape=circle, fill=black] {};
\node (6) at (60,0) [shape=circle, fill=black] {};
\node (7) at (70,0) [shape=circle, fill=black] {};

\node (6-) at (60,-10) [shape=circle, fill=black] {};
\node (5+) at (50,10) [shape=circle, fill=black] {};
\node (7+) at (70,10) [shape=circle, fill=black] {};

\draw [thick,dashed] (35,0) to (40,0);
\draw [thick] (40,0) to (80,0);
\draw [thick,dashed] (80,0) to (85,0);

\draw [thick] (5) to (5+);
\draw [thick] (6) to (6-);
\draw [thick] (7) to (7+);

\draw [thick] (6-) to (57,-10) to (56.5,-10.5) to (56.5,-11) to (57,-11.5) to (57.5,-11.5) to (6-);
\draw [thick] (6-) to (63,-10) to (63.5,-10.5) to (63.5,-11) to (63,-11.5) to (62.5,-11.5) to (6-);
\draw [thick, dashed] (58,-12) arc (225:315:3);

\draw [thick] (5+) to (47,10) to (46.5,10.5) to (46.5,11) to (47,11.5) to (47.5,11.5) to (5+);
\draw [thick] (5+) to (53,10) to (53.5,10.5) to (53.5,11) to (53,11.5) to (52.5,11.5) to (5+);
\draw [thick, dashed] (52,12) arc (45:135:3);

\draw [thick] (7+) to (67,10) to (66.5,10.5) to (66.5,11) to (67,11.5) to (67.5,11.5) to (7+);
\draw [thick] (7+) to (73,10) to (73.5,10.5) to (73.5,11) to (73,11.5) to (72.5,11.5) to (7+);
\draw [thick, dashed] (72,12) arc (45:135:3);

\draw [thick] (5) to (47,-1) to (46.5,-1.5) to (46.5,-2) to (47,-2.5) to (47.5,-2.5) to (48.2,-2) to (5);
\draw [thick] (5) to (53,-1) to (53.5,-1.5) to (53.5,-2) to (53,-2.5) to (52.5,-2.5) to (51.8,-2) to (5);
\draw [thick, dashed] (48,-3) arc (225:315:3);

\draw [thick] (6) to (57,-1) to (56.5,-1.5) to (56.5,-2) to (57,-2.5) to (57.5,-2.5) to (58.2,-2) to (6);
\draw [thick] (6) to (63,-1) to (63.5,-1.5) to (63.5,-2) to (63,-2.5) to (62.5,-2.5) to (61.8,-2) to (6);
\draw [thick, dashed] (58,-3) arc (225:315:3);

\draw [thick] (7) to (67,-1) to (66.5,-1.5) to (66.5,-2) to (67,-2.5) to (67.5,-2.5) to (68.2,-2) to (7);
\draw [thick] (7) to (73,-1) to (73.5,-1.5) to (73.5,-2) to (73,-2.5) to (72.5,-2.5) to (71.8,-2) to (7);
\draw [thick, dashed] (68,-3) arc (225:315:3);

\node at (50.5,16) {${\mathcal F}_n$};
\node at (60,1.5) {$v$};
\node at (70.5,16) {${\mathcal F}_{n+1}$};

\end{tikzpicture}
\end{center}
\caption{A bulb inserted between two flowers} 
\label{bulb}
\end{figure}

Now suppose that, as shown in Figure~\ref{bulb}, we modify $\mathcal N'_p$ by inserting an additional bulb (copy of $\mathcal B_0$) at a new vertex $v$ on the horizontal axis between adjacent flowers $\mathcal F_n$ and $\mathcal F_{n+1}$, including $l-1$ loops at $v$, below the axis, to ensure that $v$ has valency $p$; then the labelling is changed, and in particular all labels in flowers $\mathcal F_m$ for $m>n$, including their principal labels, are increased by $1$ because of the extra edge on the horizontal axis. By inserting various numbers of bulbs between flowers, we can arrange that, for each integer $d\ge 1$, every congruence class in $\mathbb Z_d$ is represented by the principal label $\lambda(n)$ of some flower $\mathcal F_n$. For example, we could deal with the gaps between successive flowers from left to right, at each stage inserting enough bulbs so that, in increasing order of $d$, all classes in $\mathbb Z_d$ have been represented. Moreover, by using additional redundant bulbs at arbitrary stages, we can do this in uncountably many different ways. As a result, we obtain uncountably many non-isomorphic maps giving transitive permutation representations of $\Gamma$, and hence uncountably many conjugacy classes of point-stabilisers in $\Gamma$, all with the required generating sets.

It remains to prove that these subgroups are all maximal, or equivalently, that the permutation representations are primitive. Any $\Gamma$-invariant equivalence relation $\sim$ on $\Omega$ restricts to the labels on $C$ as congruence mod~$(n)$ for some $n\in{\mathbb N}\cup\{\infty\}$, where we allow $n=1$ and $\infty$ to represent the universal and identity relations on $\mathbb Z$. If we define $W:=Z^{l+1}Y\in\Gamma$, inducing $w:=z^{l+1}y\in G$ on $\Omega$, then by inspection $\beta w=\beta$ whenever $\beta$ is the principal directed edge of a flower, so if $n\ne\infty$ then by our choice of gaps, $w$ preserves every equivalence class appearing in $C$. However, inspection of Figure~\ref{mapN'p} shows that $w$, in its induced action on integer labels, sends $1$ to $-2l-2$, and $-l-1$ to $l+1$, so $n$ divides both $2l+3$ and $2l+2$, giving $n=1$. Thus all elements of $C$ are equivalent under $\sim$, and hence so are all elements of $Cy$. But $\Omega=C\cup Cy$ with $\alpha\in C\cap Cy$, so $\sim$ is the universal relation.

Thus we may assume that $n=\infty$, so all elements of $C$ are in different classes, and hence the same applies to $Cy$. Since $\alpha\in C\cap Cy$ we have $E\cap C=E\cap Cy=\{\alpha\}$. Since $\Omega=C\cup Cy$ it follows that $E=\{\alpha\}$, so all equivalence classes are singletons and $\sim$ is the identity relation.
\end{proof}

The condition that $p$ should be odd is essential here: the parabolic elements are all contained in the normal closure of $Z$, which is a proper subgroup of $\Gamma$ if $p$ is even. This suggests the following:

\begin{conj}
Theorem~\ref{Hecketorsionfree} extends to the groups $\Delta(p,q,\infty)$ for all mutually coprime pairs of integers $p, q>1$, that is, each of these groups has uncountably many conjugacy classes of torsion-free maximal subgroups, each freely generated by a countably infinite set of parabolic elements.
\end{conj}

As in the case $q=2$ above, the coprimality condition is necessary here; whether it is also sufficient is not clear. A proof along the lines developed in this paper would require the construction of a map similar to the map $\mathcal N_{p,q}$ used to prove Theorem~\ref{q3}, but with the black and white vertices all of valency $p$ and $q$ respectively, and none of valency $1$.


\section{Neumann's construction revisited}\label{Neumannrev}

Neumann's construction of uncountably many Neumann permutations in~\cite{Neu33}, summarised by Magnus in~\cite[\S III.4]{Mag74}, is rather complicated, but it may be useful to restate it in terms of maps.

Neumann starts with an arbitrary sequence $h=(h_l)_{l\ge 1}$ of terms $h_l=0$ or $1$, and defines a sequence $(g_l)_{l\ge 1}$ by
\[g_1=2,\quad g_l=6l-5\sum_{i=1}^{l-1}h_i-4\quad(l\ge 2).\]
Thus $g_{l+1}-g_l=6$ or $0$ as $h_l=0$ or $1$, so every integer $g\ge 2$ can be expressed uniquely in one of the forms
\[g=g_l,\quad{\rm or}\quad g=g_l+\sigma \quad{\rm where}\quad\sigma\in\{1,2,3,4,5\}\quad{\rm and}\quad h_l=0.\]
Neumann then uses this to define a permutation $\beta$ of $\mathbb Z$, corresponding to our $x^{-1}$, and checks that it satisfies the conditions equivalent to $y:=(zx)^{-1}$ being what we have called a Neumann permutation of type $(3,2)$.

In order to define his permutation $\beta$, Neumann uses a long list of equations~\cite[(68)]{Neu33}, summarised later by Magnus in~\cite[Table~3.1]{Mag74} where his $f$ corresponds to our $y$. Instead we will define our permutation $y$ (and hence also $x$ and $z$) using the Neumann map $\mathcal N_h$ in Figure~\ref{Neumannmap}. This is constructed as follows. There are two free edges at the leftmost vertex, corresponding to $y$ fixing $0$ and $-1$, while the third edge at that vertex shows that $y$ transposes $1$ and $-2$. As we move to the right along the horizontal axis, $\mathcal N_h$ is built up from a sequence $(B_l)_{l\ge 1}$ of adjacent blocks $B_l$, each of which is of one of the two following types, depending on the value of $h_l$:
\begin{itemize}
\item if $h_l=0$ then $B_l$ consists of three free edges above the horizontal axis, attached to successive vertices on the axis and corresponding to fixed points $g_l$, $g_l+2$ and $g_l+4$ for $y$, together with edges along the axis corresponding to transpositions $(g_l+1,-3l)$, $(g_l+3,-3l-1)$ and $(g_l+5,-3l-2)$ in $y$;
\item if $h_l=1$ then $B_l$ consists of a single edge below the horizontal axis, connecting a vertex on the axis with a vertex below it, corresponding to transpositions $(-3l,-3l-1)$ and $(g_l,-3l-2)$ in $y$, with $-3l$ fixed by $x$.
\end{itemize}

\begin{figure}[h!]
\begin{center}
\begin{tikzpicture}[scale=0.18, inner sep=0.8mm]

\node (0) at (5,0) [shape=circle, fill=black] {};
\node (3) at (30,0) [shape=circle, fill=black] {};
\node (4) at (40,0) [shape=circle, fill=black] {};
\node (5) at (50,0) [shape=circle, fill=black] {};
\node (7) at (75,0) [shape=circle, fill=black] {};
\node (7-) at (75,-10) [shape=circle, fill=black] {};

\draw [thick] (5,0) to (15,0);
\draw [thick,dashed] (18,0) to (22,0);
\draw [thick] (25,0) to (60,0);
\draw [thick,dashed] (63,0) to (67,0);
\draw [thick] (70,0) to (82,0);
\draw [thick,dashed] (85,0) to (89,0);

\draw [thick] (5,-10) to (5,10);
\draw [thick] (3) to (30,10);
\draw [thick] (4) to (40,10);
\draw [thick] (5) to (50,10);
\draw [thick] (7) to (7-);

\node at (12,-1.5) {$-2$};
\node at (7,-5) {$-1$};
\node at (3.5,5) {$0$};
\node at (7,1.5) {$1$};

\node at (28.5,5) {$g_l$};
\node at (34,1.5) {$g_l+1$};
\node at (36.5,5) {$g_l+2$};
\node at (44,1.5) {$g_l+3$};
\node at (46.5,5) {$g_l+4$};
\node at (54,1.5) {$g_l+5$};
\node at (25.5,-1.5) {$-3l+1$};
\node at (36.5,-1.5) {$-3l$};
\node at (45.5,-1.5) {$-3l-1$};
\node at (55.5,-1.5) {$-3l-2$};

\node at (78,1.5) {$g_l$};
\node at (70.5,-1.5) {$-3l+1$};
\node at (72.5,-7) {$-3l$};
\node at (79,-4) {$-3l-1$};
\node at (82,-1.5) {$-3l-2$};

\node at (40,-7) {$B_l\;(h_l=0)$};
\node at (75,7) {$B_l\;(h_l=1)$};

\end{tikzpicture}
\end{center}
\caption{The Neumann map $\mathcal N_h$, with blocks $B_l$ for $h_l=0$ and $1$} 
\label{Neumannmap}
\end{figure}
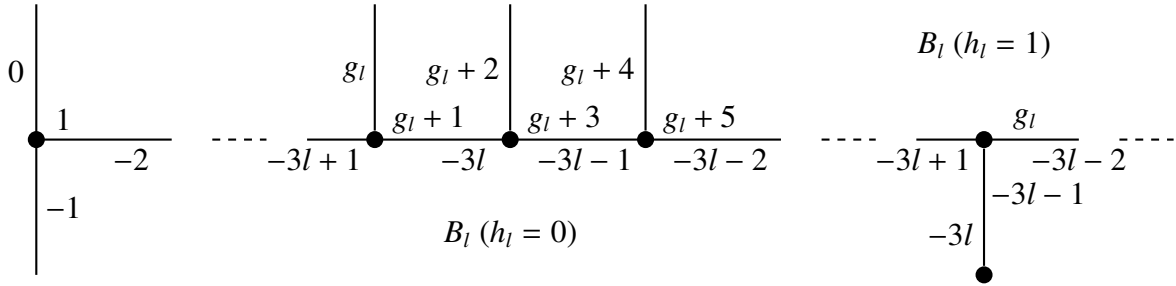

The meaning of $g_l$ should now be clear from Figure~\ref{Neumannmap}: it is the lowest positive label appearing in the $l$-th block $B_l$, corresponding to the term $h_l$ in the chosen sequence $h=(h_l)$. The $2^{\aleph_0}$ possible sequences $h$ give mutually non-isomorphic maps $\mathcal N_h$, and hence give  $2^{\aleph_0}$ conjugacy classes of Neumann subgroups $M_h$ of the modular group $\Gamma$. They are maximal nonparabolic subgroups, each a free product of cyclic groups of order $3$ and $2$ corresponding to the fixed points of $x$ and $y$. In fact, in all cases except one this construction yields a conjugacy class of maximal subgroups of $\Gamma$:

\begin{theo}
The Neumann subgroup $M_h$ is a maximal subgroup of the modular group $\Gamma$ if and only if $h$ is not the constant sequence $(0)$ given by $h_l=0$ for all $l$.
\end{theo}

\begin{proof} Suppose that $h\ne(0)$, so that $X$ has a fixed point. If $M_h$ is not maximal, then $\Gamma$ acts imprimitively on $\mathbb Z$, preserving congruence mod~$(n)$ for some integer $n\ge 2$. Let $\theta:\mathcal N_h\to\overline{\mathcal N}_h$ be the projection onto the corresponding quotient map, reducing labels mod~$(n)$. Since $X$ and $Y$ have fixed points in $\mathcal N_h$, they also have fixed points in $\overline{\mathcal N}_h$. A case-by-case argument shows that $\theta$ sends non-trivial cycles of $X$ and $Y$ in $\mathcal N_h$ to non-trivial cycles in $\overline{\mathcal N}_h$, so it induces an unbranched covering of the embedded graphs. For example, let $e$ be an edge of $\mathcal N_h$ on the horizontal axis, separating two free edges (copies of $B_l$ where $h_l=0$); let $i$ and $j=iY$ be the positive and negative labels of the directed edges in $e$, and suppose that $i\equiv j$ mod~$(n)$. Now $Y$ fixes $iX$ and $jX^{-1}$, so it fixes the congruence classes $[i]X$ and $[j]X^{-1}=[i]X^{-1}$ as well as $[i]$; these three directed edges thus form an orbit of $\Gamma$, and therefore give all the directed edges of $\overline{\mathcal N}_h$. However, there has to be a fixed point of $X$ in $\overline{\mathcal N}_h$, so they are all equal, and $\overline{\mathcal N}_h$ is trivial, a contradiction. Thus $i\not\equiv j$, so $e$ is mapped isomorphically into $\overline{\mathcal N}_h$. The other edges are dealt with similarly, as are the trivalent vertices. Hence $\theta$ induces an unbranched covering of graphs, so in particular the subgraph consisting of the leftmost vertex $v$ of $\mathcal N_h$ and its incident edges is mapped isomorphically into $\overline{\mathcal N}_h$, and its image must lift to more than one (in fact infinitely many) isomorphic copies of it in $\mathcal N_h$; however, $v$ is the only vertex in $\mathcal N_h$ incident with two free edges, a contradiction. Thus $M_h$ is maximal if $h\ne(0)$.

This argument fails if $h=(0)$ since then $X$ has no fixed points in $\mathcal N_h$, and one can form a quotient $\overline{\mathcal N}_h$, with one vertex and three free edges, by reducing labels mod $(3)$. In this case $M_h$ is a free product of infinitely many copies of $C_2$, each corresponding to a fixed point of $Y$; it is not maximal since it is contained in the normal closure of $Y$, a subgroup of index $3$ in $\Gamma$. (This is, in fact, the first example presented by Neumann in~\cite[\S1]{Neu33}: see the last comment in his \S15.)
\end{proof}

\section{Tretkoff's construction revisited}\label{Tretkoffrev}

In~\cite{Mag73} Magnus conjectured that every maximal nonparabolic subgroup of the modular group $\Gamma$ arises as a subgroup $M_h$ from Neumann's construction in~\cite{Neu33}; this was later phrased as a question rather than a conjecture in~\cite[\S III.4]{Mag74}. In~\cite{Tre}, C.~Tretkoff disproved the conjecture by finding examples of Neumann subgroups which do not arise in this way. In fact, it is clear from the discussion in the preceding section that the subgroups $M_h$ form a very small subset of the set of all Neumann subgroups: one can construct examples of the latter by starting with the semi-infinite path along the horizontal axis in Figure~\ref{Neumannmap}, and then attaching finite rooted binary plane trees, possibly with free ends, to the vertices, two at the leftmost vertex and one at each of the other vertices, on either side of the axis. There are no restrictions, as there are for Neumann's maps $\mathcal N_h$, that the attached trees should each have just one edge, that the free and non-free edges should be respectively above and below the axis, or that the free edges should come in blocks of three.

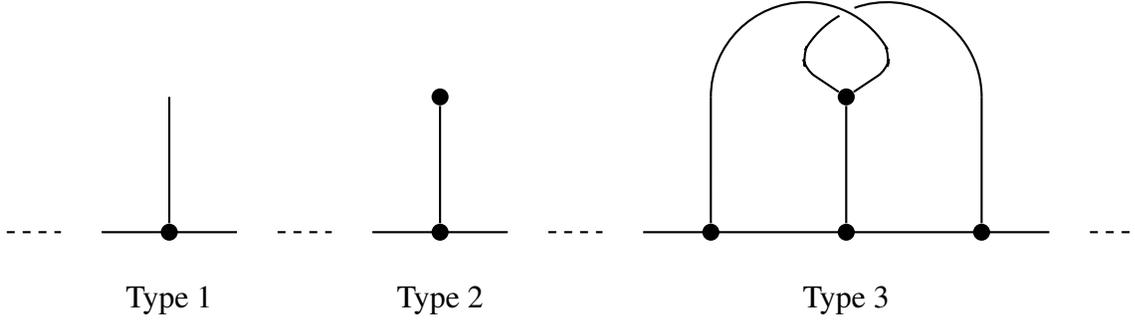
\begin{figure}[h!]
\begin{center}
\begin{tikzpicture}[scale=0.18, inner sep=0.8mm]

\node (0) at (0,0) [shape=circle, fill=black] {};
\node (2) at (20,0) [shape=circle, fill=black] {};
\node (2+) at (20,10) [shape=circle, fill=black] {};
\node (4) at (40,0) [shape=circle, fill=black] {};
\node (5) at (50,0) [shape=circle, fill=black] {};
\node (5+) at (50,10) [shape=circle, fill=black] {};
\node (6) at (60,0) [shape=circle, fill=black] {};

\draw [thick,dashed] (-12,0) to (-8,0);
\draw [thick] (-5,0) to (5,0);
\draw [thick,dashed] (8,0) to (12,0);
\draw [thick] (15,0) to (25,0);
\draw [thick,dashed] (28,0) to (32,0);
\draw [thick] (35,0) to (65,0);
\draw [thick,dashed] (68,0) to (72,0);

\draw [thick] (0) to (0,10);
\draw [thick] (2) to (2+);
\draw [thick] (4) to (40,10);
\draw [thick] (5) to (5+);
\draw [thick] (6) to (60,10);

\draw [thick] (60,10) arc (0:110:7);
\draw [thick] (49.5,16) arc (120:150:7);
\draw [thick] (40,10) arc (180:30:7);

\draw [thick, rounded corners] (53,13.5) to (53.2,13) to (53,12) to (5+);
\draw [thick, rounded corners] (47,13.5) to (46.8,13) to (47,12) to (5+);

\node at (0,-5) {Type 1};
\node at (20,-5) {Type 2};
\node at (50,-5) {Type 3};

\end{tikzpicture}
\end{center}
\caption{Tretkoff's construction of Neumann maps, with blocks of type $1$, $2$ and $3$} 
\label{Tretkoff}
\end{figure}

Like Neumann and Magnus, Tretkoff defines her subgroups in terms of permutations, but again it is instructive to reinterpret them in terms of maps. Any subgroup of $\Gamma=C_2*C_3$ is a free product of subgroups $C_r$ for $r=2, 3$ or $\infty$; as we have seen, Neumann's subgroups $M_h$ have only $C_2$ and $C_3$ as free factors, whereas Tretkoff's can have factors of all three types. As in Neumann's case, the corresponding maps are constructed from a sequence of blocks. Her main ingredients are of three types, shown in Figure~\ref{Tretkoff}. The first two, corresponding to her permutation patterns 1 and 2 and contributing free factors $C_2$ and $C_3$ to $M$, are obvious, but the third, corresponding to her pattern~3, requires some explanation: a bridge, or handle, is added to the surface to avoid the apparent crossing of edges. One can easily check that this results in a Neumann map (which is now nonplanar), and that each block of this type contributes a free factor $F_2=C_{\infty}*C_{\infty}$ to $M$. Different sequences of blocks of these three types yield $3^{\aleph_0}=2^{\aleph_0}$ conjugacy classes of maximal nonparabolic subgroups. In fact, by careful arrangement of the blocks one can use Lemma~\ref{notinv}, as in the proof of Corollary~\ref{p3}, to obtain $2^{\aleph_0}$ conjugacy classes of nonparabolic maximal subgroups.


\section{Brenner and Lyndon's construction revisited}\label{B&Lrev}

In~\cite{BL83MA} (see also~\cite[\S 2]{Lyn}) Brenner and Lyndon went further and found examples of maximal nonparabolic subgroups of the modular group $\Gamma$ which are not Neumann subgroups. Again, their construction can be explained by using maps, namely quotients of the Petrie dual (explained below) of the trivalent tessellation $\mathcal M=\{6,3\}$ of the euclidean plane by regular hexagons. They first construct a nonparabolic normal subgroup $N$, which is maximal among all such subgroups, though not itself maximal parabolic; they then classify the maximal nonparabolic subgroups containing $N$, showing that these form a countably infinite set of conjugacy classes. None of these subgroups is maximal in $\Gamma$.

The construction in~\cite{BL83MA} is as follows. The isometry group of $\mathcal M$ is the extended triangle group $\Delta=\Delta[3,2,6]$, equivalently the $2$-dimensional euclidean crystallographic group $p6m$. It has three subgroups of index~$2$.
Two of these are obvious: the orientation-preserving subgroup, which is the triangle group $\Delta(3,2,6)=p6$, and the extended triangle group $\Delta[3,3,3]=p3m1$, generated by reflections in the sides of an equilateral triangle with a side along an edge of $\mathcal M$. The latter is also the subgroup of $\Delta$ preserving the black and white bipartite colouring of the vertices of $\mathcal M$. However, there is a third subgroup of index $2$, namely the `mixed' subgroup $Q=p31m$  consisting of the elements of $\Delta$ which either preserve the orientation and the colouring (forming the subgroup $\Delta(3,3,3)=p3$ of index $4$ in $\Delta$), or reverse them both; this subgroup has the form
\[Q=\langle R, S\mid R^2=S^3=(RS^{-1}RS)^3=1\rangle,\]
where $R$ is the reflection in an axis through the midpoints of two opposite edges of a hexagonal face, and $S$ and $RS^{-1}R$ are rotations about the vertices incident with one of those edges. 

This presentation of $Q$ shows that there is an epimorphism $\Gamma\to Q$ given by $X\mapsto S, Y\mapsto R$; its kernel is the normal closure $N$ of $(YX^{-1}YX)^3$ in $\Gamma$. By using coset diagrams, Brenner and Lyndon show that $N$ is nonparabolic, and is maximal among all nonparabolic normal subgroups of $\Gamma$. By considering the (well-known) subgroups of $Q$, they describe all the nonparabolic subgroups containing $N$, and in particular they identify the maximal nonparabolic subgroups among them. There are $\aleph_0$ such subgroups, and they show that none of them is a Neumann subgroup.

If we reinterpret their work in terms of maps, then $N$ corresponds to the Petrie dual $\mathcal N$ of the map $\mathcal M=\{6,3\}$. This is an embedding of the same graph as $\mathcal M$, but with the hexagonal faces replaced with new faces, following the Petrie paths of $\mathcal M$. These are zig-zag paths, alternately turning left and right at successive vertices; there are three parallel families of them, all of infinite length, so $\mathcal N$ has three mutually disjoint families of faces, all with infinitely many sides. A subgroup of $\Gamma$ containing $N$ is nonparabolic if and only if it induces no translations of $\mathcal M$ in the three directions of the Petrie paths, so that the corresponding quotient map of $\mathcal N$ has no finite faces. In all such cases there is more than one face, so the subgroup is not a Neumann subgroup.

By adapting our earlier constructions we can extend their results to give uncountably many nonparabolic maximal subgroups of $\Gamma=\Delta(p,q,\infty)$, for all $p\ge 3$ and $q\ge 2$, none of them Neumann subgroups. They arise from maps which have all the properties of a Neumann map, except that they have more than one face, so that the maximal parabolic subgroup $P$ is intransitive on $\Omega$ and therefore does not complement the stabilisers $M=\Gamma_{\alpha}$.

\begin{theo}\label{nonNeumann}
For each pair of integers $p\ge 3$ and $q\ge 2$ the group $\Gamma=\Delta(p,q,\infty)$ has uncountably many conjugacy classes of nonparabolic maximal subgroups which are not Neumann subgroups.
\end{theo}

\begin{figure}[h!]
\begin{center}
\begin{tikzpicture}[scale=0.155, inner sep=0.8mm]

\node (-3) at (-30,0) [shape=circle, fill=black] {};
\node (-2) at (-20,0) [shape=circle, fill=black] {};
\node (-1) at (-10,0) [shape=circle, fill=black] {};
\node (0) at (0,0) [shape=circle, fill=black] {};
\node (1) at (10,0) [shape=circle, fill=black] {};
\node (1-) at (10,-10) [shape=circle, fill=black] {};
\node (2) at (20,0) [shape=circle, fill=black] {};
\node (2+) at (20,10) [shape=circle, fill=black] {};
\node (3) at (30,0) [shape=circle, fill=black] {};
\node (4) at (40,0) [shape=circle, fill=black] {};
\node (5) at (50,0) [shape=circle, fill=black] {};
\node (6) at (60,0) [shape=circle, fill=black] {};

\draw [thick] (-34,0) to (64,0);
\draw [thick,dashed] (-38,0) to (-34,0);
\draw [thick,dashed] (64,0) to (68,0);

\draw [thick] (-3) to (-30,-10);
\draw [thick] (-2) to (-20,8.5);
\draw [thick] (-1) to (-10,-10);
\draw [thick] (0) to (0,8.5);
\draw [thick] (1) to (10,-10);
\draw [thick] (2) to (20,10);
\draw [thick] (3) to (30,-8.5);
\draw [thick] (4) to (40,10);
\draw [thick] (5) to (50,-8.5);
\draw [thick] (6) to (60,10);

\node at (2.5,1.5) {$-3$};
\node at (12.5,1.5) {$-2$};
\node at (18,3) {$-1$};
\node at (21.5,7) {$0$};
\node at (22.5,1.5) {$1$};
\node at (33,1.5) {$2$};
\node at (38.5,5) {$3$};
\node at (43,1.5) {$4$};
\node at (53,1.5) {$5$};
\node at (58.5,5) {$6$};

\node at (27,-1.5) {$-3'$};
\node at (17,-1.5) {$-2'$};
\node at (12.5,-3) {$-1'$};
\node at (8.5,-7) {$0'$};
\node at (7.5,-1.5) {$1'$};
\node at (-2.5,-1.5) {$2'$};
\node at (-7.5,-5) {$3'$};
\node at (-12.5,-1.5) {$4'$};
\node at (-22.5,-1.5) {$5'$};
\node at (-27.5,-5) {$6'$};

\node at (10,10) {$F$};
\node at (20,-10) {$F'$};

\node at (-20,10) {$?$};
\node at (0,10) {$?$};
\node at (30,-10) {$?$};
\node at (50,-10) {$?$};

\end{tikzpicture}
\end{center}
\caption{A set of trivalent maps $\mathcal N$ with two faces} 
\label{mapN}
\end{figure}
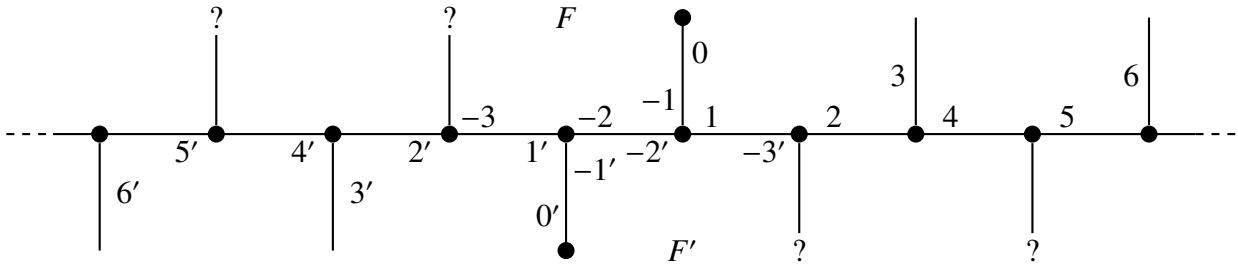

\begin{proof}
For simplicity of exposition, and to avoid repetition, we will give the proof only in the case of the modular group $\Gamma=\Delta(3,2,\infty)$. The extension to other groups $\Delta(p,q,\infty)$ is straightforward.

Figure~\ref{mapN}, which repeats in the obvious way to the right and left, represents $2^{\aleph_0}$ trivalent planar maps $\mathcal N$, where each question mark represents the possibility of either adding a $1$-valent vertex, or leaving the edge as a free edge, corresponding to a directed edge fixed by $x$ or $y$ respectively. Whatever choices are made, $\mathcal N$ represents a transitive permutation representation of $\Gamma$, in which the two faces $F$ and $F'$ correspond to the two infinite cycles $C$ and $C'$ of $z$ on the set $\Omega=C\cup C'$ of directed edges of $\mathcal N$. In each face $F$ or $F'$, a particular directed edge $\alpha$ or $\alpha'$ has been chosen, and each directed edge $\alpha z^i$ or $\alpha'z^i$ is labelled $i$ or $i'$, so that $z$ acts on $\Omega$ (now identified with the disjoint union of two copies of $\mathbb Z$) by $i\mapsto i+1$, $i'\mapsto (i+1)'$. Note that positive labels are independent of the choice of extra vertices, whereas negative labels are not, so they are mostly omitted in Figure~\ref{mapN}. 

We need to show that the extra vertices can be allocated so that $\Gamma$ acts primitively on $\Omega$. Any $\Gamma$-invariant equivalence relation $\sim$ on $\Omega$ must restrict to each of $C$ and $C'$ as congruence mod~$(n)$ or $(n')$ for some $n, n'\in{\mathbb N}\cup\{\infty\}$, where we again  include $1$ and $\infty$ to represent the universal and identity relations on $\mathbb Z$. As in the proof of Corollary~\ref{p3}, applying Lemma~~\ref{notinv} to the pairs of directed edges labelled $0, 3n$ and $0', 3n'$ shows that we must have $n, n'\in\{1, \infty\}$.

Suppose that $n=1$, so that all elements of $C$ are equivalent. It follows that all elements of $Cy$ are equivalent; this set includes distinct elements of $C'$ ($-2y=-2'$ and $1y=-3'$, for example), so $n'\ne \infty$ and hence $n'=1$, that is, all elements of $C'$ are equivalent. Since $Cy$ includes elements of $C$ (such as $0y=-1$), all elements of $Cy$ are in the same class as those in $C$, and hence $\sim$ is the universal relation on $\Omega$. By symmetry we obtain the same conclusion if $n'=1$, so we may assume that $n=n'=\infty$, that is, $\sim$ restricts to the identity relation on $C$ and on $C'$.

If $\sim$ is not the identity relation on $\Omega$, then each equivalence class must have size $2$, consisting of one element from each of $C$ and $C'$. This means that the stabiliser $M=\Gamma_{\alpha}$ of $\alpha$ in $\Gamma$ has index $2$ in the subgroup $\Gamma_E$ of $\Gamma$ preserving the equivalence class $E=[\alpha]$. Having index $2$, $M$ is normal in $\Gamma_E$, so $\sim$ is induced by a group of automorphisms $A\cong \Gamma_E/M\cong C_2$ of $\mathcal N$, transposing equivalent pairs of directed edges. The generator of $A$ must preserve the orientation of the plane, and must send free edges to free edges, and vertices to vertices of the same valency, so the only possibility is the half-turn about the centre of Figure~\ref{mapN}, transposing pairs $i$ and $i'$. However, we can eliminate this possibility by choosing the allocations of extra vertices above and below the horizontal axis so that they are not equivalent under this half-turn, for instance by allocating a black vertex at the rightmost question mark above the axis, but not at the leftmost question mark below it. This restriction still leaves us $2^{\aleph_0}$ possible allocations, and hence that number of conjugacy classes of maximal subgroups $M$ of $\Gamma$. As before, these subgroups are nonparabolic since $z$ has no finite cycles, but now they are not Neumann subgroups since $z$ has more than one cycle.
\end{proof}


\section{Consequences and generalisations}\label{conseq}

For a finitely generated group $\Gamma$, the following two properties are equivalent:

\begin{enumerate}
\item $\Gamma$ has uncountably many conjugacy classes of maximal subgroups of infinite index;
\item  $\Gamma$ has uncountably many maximal subgroups.
\end{enumerate}
Clearly (1) implies (2), and the converse depends on the facts that $\Gamma$ has only countably many subgroups of finite index, and that each subgroup of $\Gamma$ has only countably many conjugates. This equivalence remains valid even if one restricts attention to maximal subgroups satisfying some condition invariant under conjugation, such as being torsion-free, nonparabolic, etc.

If $\tilde\Gamma\to\Gamma$ is an epimorphism of groups, and $\Gamma$ has property (1) or (2), then $\tilde\Gamma$ also has that property. Thus the fact that the modular group has property (1) implies that it is shared by many other groups, such as non-abelian free groups, and hence by surface groups of genus $g>1$, etc. This has consequences outside group theory: for instance a Riemann surface of genus $g>1$ has uncountably many inequivalent coverings of infinite degree, each with no intermediate coverings.

In this paper, we have given explicit constructions for uncountable sets of conjugacy classes of maximal subgroups $M$ in hyperbolic triangle groups $\Delta(p,q,\infty)$. These subgroups are in fact `almost maximal' in the corresponding extended triangle groups $\Gamma^*=\Delta[p,q,\infty]$, in the sense that $\Gamma$ is the only subgroup $M^*$ of $\Gamma^*$ such that $M<M^*<\Gamma^*$: any other such subgroup must contain $M$ with index $2$, and induce an orientation-reversing automorphism of the corresponding map, whereas it is clear from the constructions that no such automorphism exists.

One might also consider maximal subgroups of triangle groups $\Gamma=\Delta(p,q,r)$ where $r$ is finite. Much is known about those of finite index, since the finite quotients of triangle groups have been intensively studied. However, much less seems to be known about those of infinite index.

These certainly exist. For example, by a result of Ol'shanski\u\i\/~\cite[Theorem~1]{Ols00}, any hyperbolic triangle group $\Gamma$ has a quotient $Q\ne 1$ with no proper subgroups of finite index. Since $Q$ is finitely generated,  Zorn's Lemma implies that it has maximal subgroups. These must have infinite index, so they lift back to maximal subgroups of infinite index in $\Gamma$. However, this argument gives us no information about the number of such subgroups, their structure, or the corresponding primitive permutation groups and bipartite maps.

One way of constructing specific examples is to adapt the Higman--Conder technique of `sewing coset diagrams together', used in~\cite{Con80, Con81} to construct finite alternating and symmetric quotients of $\Delta(2,3,r)$ and  $\Delta[2,3,r]$ for integers $r\ge 7$.

\begin{prop}\label{23r}
If $r\ge 7$ then the groups $\Delta(2,3,r)$ and $\Delta[2,3,r]$ each have uncountably many conjugacy classes of maximal subgroups of infinite index.
\end{prop}

\begin{proof}
If $r=7$ one can use Conder's coset diagrams $G$ and $H$ in~\cite{Con80} to form an infinite diagram
$H(1)G(1)G(1)G(1)\cdots$,
where $(1)$ denotes $(1)$-composition, and then by $(1)$-compositions attach a copy of his diagram $A$ to each of an arbitrary subset of the copies of $G$ in this chain. This gives $2^{\aleph_0}$ inequivalent infinite transitive permutation representations of $\Delta[2,3,7]$. As in~\cite{Con80}, the `useful cycle' of length $17$ appearing in $H$ ensures that $\Delta(2,3,7)$ acts primitively in each case, so the point-stabilisers in both groups are maximal subgroups of infinite index. (Those in $\Delta(2,3,7)$ constructed in this way are free products of cyclic groups of order $2$ and $3$.) This argument can be extended to the case $r\ge 7$, with the roles of $A$, $G$ and $H$ now taken by the diagrams $V(h,d)$, $S(h,d)$ and $U(h,d)$ in~\cite{Con81}, where $r=h+6d$ with $d\in{\mathbb N}$ and $h=7,\ldots, 12$. 
\end{proof}

It is hoped to give full details of this proof in a later paper, together with some applications of maximal subgroups. It seems plausible that coset diagrams constructed by Everitt~\cite{Eve} and others could be used to extend this result to other triangle groups, and to more general Fuchsian groups. In particular, Theorem~\ref{CpCqnonpara} and Proposition~\ref{23r} suggest:

\begin{conj}
Every hyperbolic triangle group has uncountably many conjugacy classes of maximal subgroups of infinite index.
\end{conj}

Each conjugacy class of maximal subgroups constructed in this paper corresponds to a primitive permutation representation of infinite degree of some triangle group, acting as the monodromy group of the associated map or hypermap. It would be interesting to know more about these representations. For example, are they faithful, and if not, what are their kernels? Are they multiply transitive, and if not, which relations do they preserve?

As a simple example, if in \S\ref{Neumannrev} we take $h_l=1$ for all $l\ge 1$, then we obtain a doubly transitive representation of the modular group $\Gamma$, since (as is easily verified) the subgroup $\langle Y,\, Z^{-1}YZ,\, Z^{-3}XZ^3\rangle$ of the stabiliser $M_h$ of $0$ acts transitively on $\mathbb Z\setminus\{0\}$. How typical is this?

\bigskip

\thanks{{\bf Acknowledgement.} The author thanks Gabino Gonz\'alez-Diez for asking a question, answering which gave rise to this investigation, and Ashot Minasyan for drawing his attention to~\cite{Ols00}.}


\end{document}